\documentclass[letterpaper, 10 pt, conference]{ieeeconf}
\IEEEoverridecommandlockouts
\usepackage[T1]{fontenc}   
\usepackage{amsmath}
\usepackage{amssymb}   
\usepackage{bbm,dsfont}
\usepackage{upgreek}
\usepackage{theorem}
\usepackage{color}
\usepackage{comment}

\theoremstyle{plain}

\newtheorem{proposition}{Proposition}
\newtheorem{corollary}{Corollary}
\newtheorem{theorem}{Theorem}
\newtheorem{lemma}{Lemma}

\usepackage{color}
\usepackage{multirow}
\usepackage{graphicx}
\usepackage{subfigure}

\newtheorem{conj}{Conjecture}

\begin{document}

\title{\Large\bf Algebraic Test for Asymptotic Stability of Periodic Orbits for Polynomial Systems}
\author{
Rafa{\l} Wisniewski and Tom N{\o}rgaard Jensen%
\thanks{The authors are with Department of Electronic Systems, Aalborg
  University, Fredrik Bajers vej 7c, DK-9220 
  Aalborg, Denmark, Email: \{raf,tnj\}@es.aau.dk}%
}

\maketitle

\begin{abstract}
We will address the problem of determining existence and asymptotic stability of a non-trivial periodic orbit in dynamical systems described by polynomial vector fields. To this end, we will lean upon the celebrated results of Borg~\cite{Borg60}, Olech and and Hartman~\cite{art:HartmanOlech.62} and newer results of Giesl~\cite{Giesl2004643}, who all employ the concept of contraction for this purpose. Importantly, we formulate a numerically tractable algebraic test. The developed algorithm is illustrated in a numerical example.
\end{abstract}

\section{Introduction}
In this paper, we develop efficient methods for determining existence and stability of critical behaviours in dynamical
systems described by polynomial vector fields. In particular, we are
interested in developing algorithms for evaluating these properties for
systems exhibiting non-trivial periodic orbits.

We strive to find sufficient and if possible also necessary conditions for asymptotic and exponential stability of periodic orbits. An indispensable property of the criterion in sight is its ability to be numerically verifiable. Obviously, the sufficiency of a criterion imposes conservativeness of the method, but also the implemented algorithm may considerably contribute to conservatism.  
To illustrate a well known situation, the Lyapunov theory provides sufficient conditions for stability of an equilibrium. Specifically, to determine asymptotic stability, one certifies whether a certain function $V:\mathbb{R}^n\to\mathbb{R}$ and its negative orbital derivative are both positive semi-definite. A case in point is that a sufficient condition for $V > 0$ is that $V$ is a sum of squares of polynomials (or colloquially  sum of squares or even SOS). In other words, checking for $V$ being an SOS is conservative as positiveness is not equivalent to being an SOS for multivariable polynomials of more than 3 indeterminates and with degree greater than 2. 
Nonetheless, if the domain of attraction of the equilibrium is a compact semialgebraic set $\{x \in \mathbb R^n|~ g(x) \geq 0\}$ then $V$ is a polynomial and according to the Putinar Stellensatz \cite{Putinar93}, the condition $V > 0$ is equivalent to the existence of a real $\epsilon > 0$ and SOSs $s_0$ and $s_1$ such that $V = \epsilon + s_0 + s_1 g $. We consider this problem numerically tractable since for a fixed degree of $V$, checking if a polynomial is SOS boils down to solving a semidefinite optimisation problem~\cite{parrilo2003semidefinite}.  


The criteria which we are interested in are ones that  can be evaluated locally in the state-space by determining certain properties of the Jacobian matrix $\frac{\partial}{\partial x} f$ of the vector field $f$. Again, to illustrate, the  Markus-Yamabe Conjecture (only true for dynamical systems of dimension $n\leq 2$) is that if $\frac{\partial}{\partial x} f(x)$ is Hurwitz for any $x \in \mathbb R^n$ then the equilibrium solution to the system (if it exists) is globally asymptotically stable (GAS). We strive to modify this criterion and extend it to study limit cycles. In general, we shall see that one needs a combination of two conditions fulfilled to draw conclusions about the existence and stability of non-trivial periodic orbits: (i) existence of an invariant set $K\subset\mathbb{R}^n$ and (ii) some contraction property of the Jacobian matrix fulfilled on $K$. In this work, we will focus on verification of the property (ii), leaving the verification of property (i) to existing numerical methods, such as those recalled in \cite{art:giesl.hafstein.13} and references therein.
As such, this work is related to algorithms determining contraction
properties in nonlinear dynamical systems \cite{art:giesl.hafstein.13}, \cite{art:parrilo.slotine.07}. 

Specifically, we will establish an algorithm for constructing a polynomial
Riemannian metric for verifying 
Borgs criterion (\cite{Borg60} and \cite{art:stenstrom.62}) on existence and stability of periodic orbits in nonlinear
dynamical systems. This is similar to \cite{Giesl2004643}, who states that a $C^1$-smooth dynamical system $\dot x = f(x)$ has an exponentially stable periodic orbit if and only if there is a Riemannian metric tensor $G$ such that 
\begin{align} 
\label{GieslCondition} 
\max \left\{ w^{\mathrm T} \left( G(x) \frac{\partial}{\partial x} f(x) + \frac{1}{2} \dot G(x) \right) w \right\} < 0
\end{align}
with $\dot G$ denoting the orbital derivative of $G$, on the semialgebraic set $S$ defined by
\begin{align*}
S \equiv \{ (x,w) \in K \times \mathbb R^{n}|~ w^{\mathrm T} G(x) w = 1,~w^{\mathrm T} G(x) f(x) = 0 \},
\end{align*} 
where $K$ is a compact, connected, positively invariant (semialgebraic) set $K$ which contains no equilibrium. 
Nonetheless, the criterion \eqref{GieslCondition} is not numerically tractable as the semialgebraic set $S$ is defined in terms of the unknown Riemannian metric $G$. Therefore, the condition \eqref{GieslCondition} in this work will be reformulated to an equivalent criterion that there exists a Riemannian metric tensor $G$ such that 
\begin{align}
\label{RafTom}
\max \left\{ w^{\mathrm T} \left( \frac{\partial}{\partial x} f(x) G(x) - \frac{1}{2} \dot G(x) \right) w \right\} < 0
\end{align}
on 
\begin{align*}
S' \equiv \{ (x,w) \in K \times \mathbb R^{n}|~ w^{\mathrm T} w = 1,~w^{\mathrm T} f(x) = 0 \}.
\end{align*} 

Notice that the definition of $S'$ does not involve the unknown metric tensor $G$. It will be shown that 
for $K = \{x \in \mathbb R^n|~q_i(x) \geq 0,i=1,...,m\}$,  the criterion \eqref{RafTom} is equivalent to existence of a polynomials $p_1$, $p_2$, and SOSs $s_1, \hdots, s_m$ such that 
\begin{align}
\label{CriterionAlgebraicModified}
- w^{\mathrm T} \left( \frac{\partial f}{\partial x} G - \frac{1}{2} \dot G \right) w + p_1(w^{\mathrm T} w -1) + p_2 w^{\mathrm T} f - \sum_{i=1}^m s_i q_i 
\end{align}
is an SOS; in \eqref{CriterionAlgebraicModified}, we have suppressed the arguments of the functions which will be done throughout the paper whenever convenient. The criterion \eqref{CriterionAlgebraicModified} can be solved by means of semidefinite optimization whose stopping criterion is the degree of the involved polynomials.
In \cite{art:parrilo.slotine.07}, the construction of
polynomial metrics is likewise addressed using SOS for determining contraction properties in
nonlinear dynamical systems described by polynomial vector
fields. This is very similar to the work addressed here. However,
\cite{art:parrilo.slotine.07} do not take periodic orbits into account.


The layout of the paper is as follows. First, we recall the formulation of the so-called
virtual dynamics (first variation) of the nonlinear system as typically encountered in
contraction theory in which the Jacobian matrix of the vector field
enters linearly. Subsequently, we recall a number of classical results that
give conditions on the Jacobian matrix, which are sufficient for
existence and asymptotic stability of non-trivial periodic orbits in the
original nonlinear system. This gives foundation for the main result of this paper - if the system is asymptotically contractible on an positive invariant set with no equilibrium, then the system has an asymptotically stable periodic orbit.  Following the account of the
necessary and sufficient conditions for establishing existence and exponential stability of a
limit cycle, we construct an SOS program which can
verify this condition numerically. Lastly, we give a numerical example
and the results of a test of the developed algorithm on this example.


\section{Preliminaries}
We consider a  nonlinear dynamical system described by the
following deterministic and time invariant ODE
\begin{equation}\label{eq:system}
\dot{x}=f(x),
\end{equation}
where $x\in U$ with
  $U$ an open subset of the Euclidean space $\mathbb{R}^n$. However, the arguments in this exposition could be extended to a Riemannian manifold following the arguments of \cite{art:HartmanOlech.62} or \cite{art:forni.sepulchre.tac.14}. The vector field $f(\cdot)$ is assumed $C^1$; hence in particular, for every initial
condition $x_0$, the solution to \eqref{eq:system} exists and is
unique. 

In contraction theory \cite{art:forni.sepulchre.tac.14}, the time evolution of virtual displacements
$\delta x\in T_x U$ ($T_x U \approx \mathbb{R}^n$ is the tangent space of $U$ at $x$) of the state $x$ at fixed time instants is examined to draw
conclusions about the temporal evolution of the system \eqref{eq:system}. The
dynamics of the virtual displacements $\delta x$ are described by the following
\begin{equation}\label{eq:firstvar}
\delta\dot{x}=\frac{\partial }{\partial x}f(x)\delta x.
\end{equation}
We see that the Jacobian matrix $J(x) \equiv \frac{\partial }{\partial
  x}f(x)$ of the vector field enters linearly in the expression of
the evolution of the virtual displacements. So the properties of the
Jacobian must be important for the temporal evolution of the virtual
displacements. In fact, the conditions that we will
investigate in this work includes properties of the Jacobian matrix.

To illustrate, we will here recall Lewis' fundamental results \cite{Lewis51} on contraction. To this end, we will first define two concepts: (i) a Finsler metric and (ii) a variation. 

(i) A function $\psi:U \times \mathbb R^n \to \mathbb R$ smooth on $U \times \mathbb R^n \setminus \{0\}$ is a Finsler metric if
\begin{itemize}

\item it is positively homogeneous of degree $1$ in the second variable ($\psi(x,\lambda v) = \lambda \psi(x,v)$ for $\lambda > 0$), 
\item $\psi(x,w) > 0$ for all $x \in U$ and $w \in \mathbb{R}^n \setminus \{0\}$, 
\item the Hessian $[H(x,v)]_{ij} \equiv \left[\frac{\partial^2 \psi(x,v)}{\partial x_i \partial x_j}\right]$ of $\psi$ is positive definite. 
\end{itemize}
Equipped with the Finsler metric, the length $L(\gamma)$ of a sufficiently regular curve $\gamma$ is 
$$
L(\gamma) \equiv \int_0^1 \psi(\gamma(t), \dot \gamma(t) )dt.
$$
In the following, we may think of $\psi(x,v) = \sqrt{g(x)(v,v)}$ for a Riemannian metric $g$ on $U$. 

(ii) For a curve $\gamma: [0, 1] \to U$, we define a variation $\Gamma(t,s) \equiv \Gamma_{\gamma}(t,s)$ of \eqref{eq:system} such that $\Gamma(0,s) = \gamma(s)$. Furthermore, for any $s \in [0, 1]$, $\Gamma(t,s)$ is the solution of \eqref{eq:system} for the initial value at $\gamma(s)$. We are ready to recall and interpret Theorem 2 in \cite{Lewis51}.

\begin{theorem}\cite[Thm~2]{Lewis51}
\label{LewisContractionTh2}
Let $\Gamma(t,s)$ be a variation of \eqref{eq:system} and $\psi$ be a Finsler metric. Suppose that on $U$ the following contraction condition holds
\begin{equation}
\label{LewisContractionCondition}
a \leq \frac{\partial \psi(x,w)}{\partial x} f(x) +  \frac{\partial \psi (x,w)}{\partial w} J(x) w \leq A
\end{equation}
for every $x$ in the image of $\Gamma$ and $w\in T_xU$ such that $\psi(x,w) = 1$, and some $a < A$.
Then
$$
L(\Gamma(0, \cdot)) e^{a t} \leq  L(\Gamma(t,\cdot))  \leq L(\Gamma(0, \cdot)) e^{A t}. 
$$

\end{theorem}
In short, Theorem~\ref{LewisContractionTh2} states that if the contraction condition is satisfied, the distance between trajectories of the dynamic system remains bounded for finite $t$. If $a,A<0$ then trajectories (exponentially) converge to each other. In fact, Theorem~\ref{LewisContractionTh2} explains why we use the name ``contraction'', i.e.,  the trajectories (exponentially) contract to each other. The next question is then; what is the expected behavior in the limit of a system satisfying the contraction?

For planar systems ($n=2$), the investigation of critical behaviours of
\eqref{eq:system} becomes particularly simple, so we will discuss
these first.

\section{Planar systems} 
As a first step, consider the na{\"{i}}ve interpretation of the linearised dynamics in
\eqref{eq:firstvar} as an actual linear system and that as long as the Jacobian matrix
 is Hurwitz for every $x\in \mathbb{R}^n$, then the
equilibrium solution (if it exists) of the original non-linear dynamical system
\eqref{eq:system} is globally asymptotically stable (GAS). A GAS equilibrium solution to
\eqref{eq:system} will of course rule out the existence of non-trivial
periodic orbits.

This na\"{i}ve interpretation was formulated for
nonlinear time invariant systems in the so-called
Markus-Yamabe Conjecture (MYC) in \cite{art:markus.yamabe.1960} which we will recall here.
Suppose (possibly after an affine change of coordinates) that the
non-linear time invariant dynamical system \eqref{eq:system} fulfils
\begin{equation}\label{eq:autsystem}
f(0)=0\;\textnormal{and}\; x\neq 0\Leftrightarrow f(x)\neq 0.
\end{equation}
Then we have the following conjecture in \cite{art:markus.yamabe.1960}
\begin{conj}\label{conj:MYC}
Given the system \eqref{eq:system}, assume \eqref{eq:autsystem} holds
and $J(x)$ is Hurwitz for every $x\in\mathbb{R}^n$, then the
equilibrium point $x=0$ is GAS.
\end{conj}
Conjecture \ref{conj:MYC} was in Markus and Yamabes paper \cite{art:markus.yamabe.1960}
attributed to Aizerman \cite{art:aizerman.49}. However, in later works
on the subject, the conjecture is called the Markus-Yamabe Conjecture
or the Jacobian Conjecture on asymptotic stability.
Conjecture \ref{conj:MYC} is trivially true when $n=1$ and was proven in \cite{art:markus.yamabe.1960}
for the special case in which the system is triangular.

In the mid nineties, it was proven that in the special
case of $n=2$, MYC is true and at least four different proofs of MYC
for $n=2$ exists
\cite{art:fessler.1995,art:glutsuk.1995,art:gutierrez.1995,art:chen.et.al.2001}.
Around the same time, the
conjecture was ultimately proven false for systems of dimension three
and above ($n\geq 3$) where counter examples have been
found. An excellent account of the process involved in finding these
counter examples is provided in \cite{book:essen}. This process
involved finding closed form solutions to dynamical systems described
by vector fields with Hurwitz Jacobian. A number of these counter examples
are found in \cite{thesis:hubbers}. 
{Consequently, if $n\leq 2$, the equilibrium solution exists and it can be shown that the Jacobian of
$f(\cdot)$ is Hurwitz for every $x\in \mathbb{R}^n$ then one can rule out the
existence of non-trivial periodic solutions to \eqref{eq:system}.
However, if $n\geq 3$ this is no longer sufficient.}

To infer existence of non-trivial critical
behaviours in a planar system one can turn ones attention to
invariant sets in which no equilibrium solution is found. This is
described in the well known Poincar\'{e}-Bendixson Theorem (see
e.g. \cite{book:hirsch.et.al}). The
theorem shows that the compact limit sets of planar systems either
consists of equilibria or of closed orbits.
The following corollary of the Poincar\'{e}-Bendixson Theorem is particularly useful when
determining the existence of limit cycles in planar systems.
\begin{corollary}
A compact set $K$ that is positively or negatively invariant contains
either a limit cycle or an equilibrium point.
\end{corollary}
As we can see from the corollary above, a compact invariant set $K$ of the
system which does not contain an equilibrium point will contain a
limit cycle. Furthermore, solutions of the system starting in $K$ will
approach this limit cycle. Lastly, the following corollary
states that with the existence of a closed orbit in a
planar system also follows the existence of an equilibrium point.
\begin{corollary}\label{cor:pbt2}
Let $\gamma$ be a closed orbit that form the boundary of an open set
$U$. Then $U$ contains an equilibrium point.
\end{corollary}
The above corollary has the consequence that in $\mathbb{R}^2$ only
invariant sets $K$ which are not simply connected can be without
equilibria. So to infer existence of non-trivial periodic orbits one
has to investigate invariant sets which are not simply connected.
The last corollary also implies that any planar system in which the Jacobian matrix $J(x)$ is
Hurwitz for every point $x$ cannot have a periodic solution. 

The
reasoning is as follows. Suppose the planar system has a periodic
solution, then by Corollary \ref{cor:pbt2} it also has an equilibrium
solution. However, if the Jacobian is Hurwitz for every point $x$,
then this equilibrium is GAS, which rules out the existence of a
periodic solution. Thus we have a contradiction. 

As we have seen in this section, existence of non-trivial periodic
orbits in planar systems ($n=2$) can be ruled out if the Jacobian
matrix is Hurwitz for every point $x\in \mathbb{R}^2$. Furthermore, existence of non-trivial periodic orbits in
planar systems is characterised by non-simply connected invariant sets which
does not contain equilibria. 

For the
general case when $n>2$ these results are no longer valid. Instead, one
can use a combination of invariant sets and conditions on the Jacobian
matrix of the vector field similar to that in the MYC to infer
existence (and stability) of closed orbits. This is described in Borg's
Theorem.

\section{From planar systems to $\mathbb{R}^n$}
In his paper \cite{Borg60}, Borg gives sufficient conditions
for the existence and stability of closed orbits for a dynamical
system in $n$-dimensional Euclidean space. Among the conditions given in his theorem, we find the existence of a
bounded set $U$ in which the symmetric part of the Jacobian of the system is stable in all
directions orthogonal to the vector field. Additional conditions
are boundedness of the norm of the Jacobian and the vector
field itself and non-existence of equilibria. We recall Borg's Theorem here.
\begin{theorem}
\label{BorgsTeorem}{\cite{Borg60}}
Let $U\subset\mathbb{R}^n$ be connected, bounded and open and $\delta,\delta_1,k>0$ be
scalar constants. Assume that
\begin{align}
&0<\delta\leq ||f(x)||\leq k \label{NoEquilibrium}\\
&||J(x)||\leq k \label{BorgsBounded} \\
\begin{split}\label{eq:borg2}
&\left\langle w,J(x) w\right\rangle\leq -\delta_1||w||^2\\
&\textnormal{for all $x \in U$ and $w \in \mathbb{R}^n$ such that}\\
&\left\langle f(x),w\right\rangle=0
\end{split}
\end{align}
Furthermore, assume that there exists a solution
$\phi(t)$ of \eqref{eq:system} such that $\phi(t)\in U$ for every
$t\in[0,\infty)$.

Then the system \eqref{eq:system} has a unique orbitally stable\footnote{see e.g. \cite{book:jordan.smith} Definition 8.1 for a definition of orbital stability} periodic
solution contained in $U$. This solution is the limit cycle of $\phi(t)$. 
\end{theorem}
Furthermore, Borg sharpens the theorem for the particular case where the set $U$ is invariant, as expressed in the following theorem.

\begin{theorem}
Let $U$ be a region where the conditions \eqref{NoEquilibrium}, \eqref{BorgsBounded} and \eqref{eq:borg2} holds and any solution $\phi(t)$ of \eqref{eq:system} with $\phi(t_0)$ belonging to the boundary $\partial U$ for some $t_0$ satisfies: $\dot{\phi}(t_0)$ has a positive projection on the inner normal (supposed existing) of $\partial U$ at the point $\phi(t_0)$.

Then $U$ contains precisely one periodic orbit which is also the limit cycle of all solutions in $U$ (for $t\to\infty$). 
\end{theorem}
That is, when $U$ is positively invariant for the vector field $f$, and the conditions \eqref{NoEquilibrium}-\eqref{eq:borg2} are fulfilled then there is a non-trivial periodic orbit in $U$ and all solutions in $U$ converges to this orbit.

Note that the condition \eqref{eq:borg2} is a contraction
condition on the Jacobian matrix; for $\psi(x,w) = ||w||$ the contraction \eqref{LewisContractionCondition} implies the condition \eqref{eq:borg2}. 
Condition \eqref{eq:borg2} require the eigenvalues
of the symmetric part $H(x)=\frac{1}{2}(J(x)+J^T(x))$ of the Jacobian to be negative in every
direction orthogonal to the vector field at every point $x$ in
the set $U$.

In a nutshell, Borg's Theorem states that if a dynamic system has no equilibrium point (Condition \eqref{NoEquilibrium}), satisfies the contraction condition and $U$ is invariant then there is a unique asymptotically stable periodic orbit.

Seen as an extension to the Poincar\'{e}-Bendixson Theorem which was
valid for planar systems, the additional condition \eqref{eq:borg2} in
Borg's Theorem that the system needs to be contracting in directions
orthogonal to the vector field seems natural. If this was not the
case, then the set $U$ might just be occupied by a strange
attractor, say. The latter type of critical behavior is of course not
possible when the system is planar, which is why the additional
condition is not needed for planar systems. As a further extension,
Stenstr\"{o}m later generalised Borg's Theorem to a general Riemannian
manifold \cite{art:stenstrom.62}. Also worth mentioning in the contents of this paper is a generalization of Borgs theorem in \cite{art:HartmanOlech.62}, where predominantly the condition \eqref{eq:borg2} is substituted by a weaker condition \eqref{OlechWeakerCondition}.
For an arbitrary scalar product $\left<\cdot,\cdot\right>$ on $\mathbb{R}^n$ let 
	\begin{equation}
	\label{eq:gamma}
	\lambda(x) \equiv \max_{|w|=1, \left<w,f(x)\right> = 0} \left<w, J(x)w\right>,	
	\end{equation} 
and for a solution $\phi(t)$ of \eqref{eq:system} and $c \in \mathbb R$, let $\Lambda(t)$ be
$$
\Lambda_{\phi,c}(t) \equiv \int_0^t \lambda(\phi(s)) ds + ct
$$
\begin{theorem}\cite[Thm 5.3]{art:HartmanOlech.62}
\label{TheoremLimitCycleOlech}
Let $U \subset \mathbb{R}^n$ be as in Theorem \ref{BorgsTeorem} and suppose that Conditions \eqref{NoEquilibrium} and \eqref{BorgsBounded} hold. Furthermore, suppose that there is a solution $\phi(t)$
that satisfies the following two conditions
\begin{enumerate}
\item There is an $l > 0$ such that the Hausdorff distance $d(\phi(t), \partial U)$ between $\phi(t)$ and the boundary of $U$ is greater than $l$. 
\item There are $c > 0$ and $C \in \mathbb R$ such that
\begin{equation}
\label{OlechWeakerCondition}
\Lambda_{\phi,c}(t) - \Lambda_{\phi,c}(s) \leq C \hbox{ for all } s, t \hbox{ with } 0 \leq s < t < \infty.
\end{equation}
\end{enumerate}
Then there is an asymptotically stable periodic orbit.
\end{theorem}
Condition~\eqref{OlechWeakerCondition} can be interpreted as: the time average of $\lambda(\cdot)$ along the solution $\phi$ shall be bounded from above, i.e., the $\lambda(\cdot)$ might be positive on a bounded time interval, but most of the time it is negative. As a concluding remark on Theorems \ref{BorgsTeorem} and \ref{TheoremLimitCycleOlech}, the scalar product $\left< \cdot, \cdot \right>$ is the matter of choice. In particular, this makes Condition \eqref{OlechWeakerCondition} numerically intractable. 

In \cite{art:HartmanOlech.62}, we also find the following extension of the conditions in the MYC which holds for $n>2$ and as a consequence rules out the existence of closed orbits.

\begin{theorem} \cite[Thm~2.1]{art:HartmanOlech.62}
\label{HartmanOlech2.1}
Consider the dynamical system \eqref{eq:system} with a $C^1$ vector field $f$. Suppose that the equilibrium point $0$ is locally asymptotically stable (e.g. if $J(0)$ is Hurwitz) and that \eqref{eq:autsystem} holds.
The point $0$ is also GAS if 
	
$$ \lambda(x) \leq 0~~~\forall_{x \in \mathbb{R}^n}, $$
where $\lambda$ is defined in \eqref{eq:gamma}.	
\end{theorem}
At this point, we wish to attach a comment to the Markus-Yamabe conjecture. As seen above, it is generally not enough to determine if the Jacobian $J(x)$ is Hurwitz. Instead, we must study the eigenvalues of its symmetric part $H(x)$ as in \eqref{eq:gamma}. Parallel to the case of planar systems, we see that we more or less need to rule out the existence of an equilibrium point whenever the Jacobian condition (here $\lambda(x)\leq 0$ and for planar systems it was the condition that $J(x)$ is Hurwitz) is fulfilled. Otherwise, we can rule out the existence of a closed orbit. We observe that Theorem~\ref{HartmanOlech2.1} is a slight generalisation of Krasovskii's Theorem~\cite[Thm 21.1]{book:krasovskii} or \cite[Sec.30]{book:Hahn}. 

Further generalization can be established by substituting a scalar product $\left< \cdot, \cdot \right>$ by a Riemannian metric $g$. To this end, we represent $g$ on $U$ in the form $w^T G(x) w$, where $w \in T_x U \approx \mathbb{R}^n$ and $G(x)$ is positive definite on $U$. Consequently, the contraction condition becomes \cite{art:HartmanOlech.62}

\begin{align}
\label{Eq:RiemannianMetric}
	\lambda_G(x) \equiv \max_{w^{\mathrm T} G w=1, w^{\mathrm T} G f(x) = 0} w^{\mathrm T} \left( G(x) J(x)  + \frac{1}{2}  \dot G(x) \right)  w, 
\end{align}
with $G = [g_{kl}]$ and $\dot G = [f^j \partial_j g_{kl}]$, where we have used Einstein notation, i.e., the same index in the superscript and the subscript implies a summation over this index.
In the sequel, we identify the Riemannian metric $g$ with the metric tensor $G$ and even call $G$ a Riemannian metric. 

As an observant reader might have noticed most of the literature on contraction used so far stems from fifties and sixties. However, more recently Giesel \cite{Giesl2004643} formulated necessary and sufficient conditions for existence of an exponentially stable periodic orbit.


\begin{theorem}\cite[Thm 26]{Giesl2004643}
\label{LimitCycleGiesl}
The following conditions are equivalent
\begin{enumerate}
\item The system~\eqref{eq:system} has an exponentially stable periodic orbit, and the real parts of all Floquet exponents (except the trivial ones) are less than  $-c < 0$.
\item
There are a Riemannian metric $G$  and a nonempty compact, connected  subset $K$ of $\mathbb{R}^n$ that is an invariant set of the dynamical system \eqref{eq:system} containing no equilibrium such that  $\max_{x \in K}\lambda_G(x) < -c < 0$. 
\end{enumerate}
\end{theorem}

We return to the main objective of this paper - to the formulation of an algebraic criterion of existence of exponentially stable periodic orbits. As seen in Theorem~\ref{LimitCycleGiesl}, such a criterion comprises the contraction and the invariance conditions. Our focus in the sequel is on the contraction condition. Nonetheless, our task is not yet reached as the Riemannian metric $G$ is not known, and the usage of SOS program together with positive stellens{\"a}tze does not allow for an unknown in the equations of constraints - inhere $w^{\mathrm T} G f = 0$. To circumvent this problem, we formulate the following proposition.

\begin{proposition}
\label{Proposition:RemoveGfromConstrains}
Let $K$ be a compact subset of an open subset $U \subseteq \mathbb R^n$. 
Then the following three conditions are equivalent
\begin{enumerate}
\item There are $c > 0$ and a Riemannian metric $G$ on $U$ such that $ \lambda_G(x)  \leq -c$ for all $x \in K$, where $\lambda_G$ is defined in \eqref{Eq:RiemannianMetric}. 
\item There are $c > 0$ and a Riemannian metric $G$ on $U$ such that $\lambda_G'(x) \leq -c$ for all $x \in K$, where $\lambda_G'$ is defined by
\begin{align*}
\hspace{-0.5cm}
\lambda_G'(x) \equiv \max_{w^{\mathrm T} w=1, w^{\mathrm T} f(x) = 0} w^{\mathrm T} \left( J(x) G(x)  - \frac{1}{2}  \dot G(x) \right)  w, 
\end{align*}
\item There are $c' > 0$ and a Riemannian metric $G$ with polynomial entries on $U$ such that $\lambda_G'(x) \leq -c'$ for all $x \in K$. 
\end{enumerate}
\end{proposition}

\begin{proof}
The equivalence between 2) and 3) follows from Weierstrass Approximation Theorem. 

We show equivalence beetween 1) and 2). It will be instrumental to use the notation $G = [g_{kl}]$ and $G^{-1} = [g^{lj}]$. Let $\phi_z(t)$ be the flow line of $\dot x = f(x)$ with $\phi_z(0) = z$
Let $H = G^{-1}$ then specifically
\begin{align*}
0 = \left. \frac{\mathrm{d}}{\mathrm{d}t}\right |_{t=0}\left[H (\phi_z(t)) G(\phi_z(t)) \right]. 
\end{align*}
Hence
\begin{align*}
H(\phi_z(0)) \left. \frac{\mathrm{d}}{\mathrm{d}t}\right |_{t=0} [G(\phi_z(t))] = -\left. \frac{\mathrm{d}}{\mathrm{d}t}\right |_{t=0} [H(\phi_z(t))] G(\phi_z(0)).
\end{align*}
It follows that
\begin{align*}
H(z) [\partial_i g_{kl}(z) f^i(z)] H(z) = - [\partial_i g^{lj}(z) f^i(z)], 
\end{align*}
or in other words,
\begin{align*}
H(z) \dot G(z) H(z) = - \dot H(z). 
\end{align*}
Let $v \equiv G(x) w$ then 
$$ w^{\mathrm T} \left( G J + \frac{1}{2}  \dot G \right)  w = v^{\mathrm T} \left( J H   + \frac{1}{2} H \dot G H \right)  v$$
and hence
$$ \lambda_G''(x) \equiv \max_{v^{\mathrm T} H v=1, v^{\mathrm T} f(x) = 0} v^{\mathrm T} \left( J(x) H(x)  - \frac{1}{2}  \dot H(x) \right)  v \leq c.$$
Since $H$ is a Riemannian metric, $K$ is compact, and $\lambda_G''(x) \leq c, ~x \in K$, there is $c' > 0$ such that $\lambda_G' (x) \leq - c'$ on $K$.
Hence, the conclusion of the proposition follows.
\end{proof}

Combining Theorem~\ref{LimitCycleGiesl} with Proposition~\ref{Proposition:RemoveGfromConstrains} results in the following theorem 

\begin{theorem}
\label{TheoremFinalAnylyticTest}
Let $G$ be a Riemannian metric on an open subset $U$ of  $\mathbb R^n$. Let $K \subset U$ be a non-empty, compact, connected and positively invariant set of the dynamical system \eqref{eq:system}, which contains no equilibrium. Suppose that $\lambda_G'(x) < -c'$ for some $c' > 0$ for all $x \in K$, where $\lambda_G'(x)$ is given by
\begin{align}
\label{Eq:AlgebraicFormulationContraction}
\lambda_G'(x) \equiv \max_{w^{\mathrm T} w=1, w^{\mathrm T} f(x) = 0} w^{\mathrm T} \left( J(x) G(x)  - \frac{1}{2}  \dot G(x) \right)  w, 
\end{align}
Then there is a unique periodic orbit. This periodic orbit is exponentially stable.
\end{theorem}
Notice that the contraction criterion \eqref{Eq:AlgebraicFormulationContraction} does not involve Riemannian metric $G$ in the equations of constraints. In the next section, we will reformulate \eqref{Eq:AlgebraicFormulationContraction} in terms of SOSs.

\section{Algebraic Conditions}

In this section, we will evaluate the contraction criterion \eqref{Eq:AlgebraicFormulationContraction} discussed in the previous section for formulating a numerically tractable algorithm. 

The contraction criterion $\lambda_G'$ in \eqref{Eq:AlgebraicFormulationContraction}  can be computed by means of certificates of positivity provided $f$ is a polynomial vector field as seen in the next proposition. The following notation will be instrumental $\mathbb{R} [Z] \equiv \mathbb{R} [Z_1, \hdots, Z_n]$ is the ring of polynomials with real coefficients with $n$ indeterminates, $\Sigma^2[Z] \subset \mathbb{R} [Z]$ is the set of sums of squares of polynomials. The compact set $K$ brought up in Theorem~\ref{TheoremFinalAnylyticTest}, will here be a semialgebraic set of the form $K = \{x \in \mathbb{R}^n|~q_i(x) \geq 0,~i=1,...,l\}$ for some $q_i \in \mathbb{R} [X],~i=1, \hdots, l$. We define a quadratic module $Q(q_1,\hdots,q_l)$ generated by $q_1,\hdots,q_l$ as
$$
Q(q_1,\hdots,q_l) \equiv \left\{s_0+ \sum_{i=1}^l s_i q_i|~s_j \in \Sigma^2[X], j = 0, \hdots, l\right\}.
$$

In the next proposition, we will show that the contraction condition $\lambda_G'(x) < 0, x \in K$ in Theorem~\ref{TheoremFinalAnylyticTest} can be reformulated as an algebraic condition, which can be verified by an SOS program.

\begin{proposition}
\label{Prop:SOS}
Let $G = [g_{kl}]$ be a Riemannian metric with $g_{kl} \in \mathbb{R} [X]$, and $f=(f_1, \hdots, f_n)$ be a vector field with $f_i \in \mathbb{R} [X]$. For $q_i \in \mathbb{R} [X],~i=1, \hdots, l$, let $K = \{x \in \mathbb{R}^n|~q_i(x) \geq 0,~i=1,...,l\}$. Suppose  that there is $q \in Q(q_1,\hdots,q_l)$ such that $\{x \in \mathbb R^n|~q(x) \geq 0\}$ is compact.
Then the following two conditions are equivalent
\begin{enumerate}
	\item $\lambda_G'(x)  < 0$ for all $x \in K$ with $\lambda_G$ defined in \eqref{Eq:AlgebraicFormulationContraction}.
	\item  There are an $\epsilon > 0$, polynomials $p_1,~p_2 \in \mathbb R[X,W]$ and $s_i \in \Sigma^2[X,W]$, $i = 1,...,l$ such that 
\end{enumerate}    
	\begin{align}
 - & \epsilon - W^{\mathrm T} \left( J(X) G(X) - \frac{1}{2} \dot G(X) \right) W \nonumber \\
&- \sum_{i = 1}^l s_i(X,W) q_i(X) +  p_1(X,W) (|W|^2 - 1)\nonumber\\
&+ p_2(X,W) W^{\mathrm T} f(X)  \in \Sigma^2[X,W]. \label{EQ:FinalAlgebraicCondition}
	\end{align}
\end{proposition}

Specifically, the existence of $q \in Q(q_1, \hdots, q_m)$ in Proposition~\ref{Prop:SOS} is assured if there is $i \in \{1, \hdots l\}$ such that $\{x \in \mathbb R^n|~q_i \geq 0\}$ is compact. Since $K$ is compact, we can add to the family of constraints an extra constraint $q_{l+1} \geq 0$, $q_{l+1} \in \mathbb{R} [X]$, such that $C \equiv \{x \in \mathbb{R}^n|~q_{l+1}(x) \geq 0 \}$ is compact and $K \subseteq C$. 

Notice also that the Riemannian metric $G$ enters the algebraic criterion \eqref{EQ:FinalAlgebraicCondition} affinely; hence, $G$ for which $\lambda_G'$ is negative on $K$ can be computed by semidefinite programming (that maximizes $\epsilon$).  

\begin{proof}
	If \eqref{EQ:FinalAlgebraicCondition} is satisfied then $\lambda_G'(x) \leq - \epsilon < 0$ for all $x \in K$. To prove the converse, notice that any polynomial $p$ can  be written in the form
	\[p = s_1 - s_2 \]
	for $s_1$ and $s_2$ sum of squares, since $p = (p+1)^2-(p^2+1)$. We observe that for a $q \in Q(q_1, \hdots, q_l)$, we have $q\left(1 - W^{\mathrm T} W \right) \in Q\left(q_1, \cdots, q_l, W^{\mathrm T} W - 1, 1 - W^{\mathrm T}W, W^{\mathrm T} f, - W^{\mathrm T} f\right)$. Furthermore, by the hypothesis, $q\left(1 - W^{\mathrm T} W \right)$ is compact. Therefore by Putinar's Positivstellensatz~\cite{Putinar93}, we conclude that 
	\begin{align*}
		- &\epsilon - W^{\mathrm T} \left( J G - \frac{1}{2} \dot G \right) W 
        \\ = &
        s_0 + \sum_{i = 1}^l s_i q_i + (1-|W|^2) (s_{l+1} - s_{l+2}) \\ 
		+ & W^{\mathrm T} f(X)  (s_{l+3} - s_{l+4}).
	\end{align*}
	for some $s_0, \hdots, s_{l+4} \in \Sigma^2[X,W]$.
Hence, 
	\begin{align*}
  - \epsilon - & W^{\mathrm T} \left( J G - \frac{1}{2} \dot G \right) W- \sum_{i = 1}^l s_i q_i + p_1 (|W|^2 - 1) \\
 + & p_2 \left<W, f \right>  \in \Sigma^2[X,W]
	\end{align*} 
	for some $s_0, \hdots, s_l  \in \Sigma^2[X,W]$, and  $p_1, p_2 \in \mathbb R[X,W]$.
	
%
	\end{proof}

Let $V(q_i) \equiv  \{x \in \mathbb{R}^n|~ q_i(x) = 0\}$. If $0$ is a regular value of $q_i$, i.e., $Dq_i(x) \equiv \frac{\partial q_i}{\partial x} (x) $  is a surjection for all $x \in V(q_i)$ then $V(q_i)$ is a smooth manifold. Consequently, by Nagumo Theorem \cite[Thm 1.2.1]{Aubin91}, we have the following result. 
\begin{lemma}
\label{Lemma_Simple_Invariance}
Let $K = \{x \in \mathbb{R}^n|~q_i(x) \geq 0,~i=1,...,l\}$ and $f$ is a vector field on $\mathbb{R}^n$. 
Suppose that for each $i \in \{1, \cdots, l\}$, $0$ is a regular value of $q_i$, and $Dq_i[f](x) \leq 0$  for all $x \in K \cap V(q_i)$ ($Dq_i[f]$ means the differential $Dq_i$ acts as a linear map on the vector field $f$). Then $K$ is a positive invariant set of $f$.
\end{lemma}

Combining Proposition~\ref{Proposition:RemoveGfromConstrains}, Theorem~\ref{TheoremFinalAnylyticTest} , Proposition~\ref{Prop:SOS}, and Lemma~\ref{Lemma_Simple_Invariance} gives the following corollary.

\begin{corollary}
\label{CombiningCorollary}
For $q_i \in \mathbb{R} [X],~i=1, \hdots, l$, let $K = \{x \in \mathbb{R}^n|~q_i(x) \geq 0,~i=1,...,l\}$. Suppose  that there is $q \in Q(q_1,\hdots,q_l)$ such that $\{x \in \mathbb R^n|~q(x) \geq 0\}$ is compact.
Let $0$ be a regular value for $q_i$, $i = 1, \hdots l$ and $Dq_i[f](x) \leq 0$ for all $x \in K \cap V(q_i)$. Suppose that \eqref{EQ:FinalAlgebraicCondition} is satisfied for an $\epsilon > 0$, polynomials $p_1,~p_2 \in \mathbb R[X,W]$ and $s_i \in \Sigma^2[X,W]$, $i = 1,...,l$.  
Then there is a unique periodic orbit. This periodic orbit is exponentially stable.
\end{corollary}

Notwithstanding Corollary~\ref{CombiningCorollary} does not provide information about Floquet exponents, one can still conclude from Proposition~\ref{Proposition:RemoveGfromConstrains} that there is $c > 0$ such that  $\lambda_{G^{-1}}(x) < -c$ on $K$. Hence, all Floquet exponents except the trivial ones are less than or equal to $-c$. Importantly, for fixed $G$, the value of $c$ can be computed by means of the SOS programming. 

Nonetheless, our focus in this paper is on the contraction criterion, the Corollary~\ref{CombiningCorollary} provides the algebraic conditions for both the contraction and the invariance condition formulated in terms of certificates of positivity of polynomials, which can be verified by an SOS program.

\section{Numerical Example}

As an illustrative example of a system with an asymptotically stable
non-trivial periodic orbit, we consider the three a dimensional dynamical system given by
\begin{equation}\label{eq:lotka3d}
\begin{split}
\dot{x}=&x(1-x^2-y^2) (x+0.5)-y,\\
\dot{y}=&y(1-x^2-y^2)(x+0.5)+x, \\
\dot{z}=&-z.
\end{split}
\end{equation}
We have implemented the problem of maximizing an $\epsilon \in \mathbb{R}$ subject to the contraction condition \eqref{EQ:FinalAlgebraicCondition} as given in Proposition \ref{Prop:SOS} in YALMIP with MOSEK solver. The set $K$ is given by $K = \{x \in \mathbb R^n |~ 0.7 \leq |x| \leq 1.5\}$.
We have searched for a Riemannian metric with polynomial entries of degree 6. The problem was solved successfully and to illustrate, the computed $(1,1)$ entry of the Riemannian metric G is 
\begin{align*}
& g_{11}(x_1,x_2,x_3) = 12-3x_1+7x_2+14x_1^2-26x_1 x_2 \\
 &-40 x_2^2+ x_3^2 + 8 x_1^3 + 14 x_1^2 x_2 + 14x_1 x_2^2-42x_2^3 + \hdots \\
&+23 x_1^4 x_2^2 -4x_1^3 x_2^3 -2 x_1^4 x_3^2 + 2 x_1^3x_2 x_3^2 + x_1^2 x_2^2 x_3^2,
\end{align*}
where we have left out several terms for conciseness.

To underline the importance of checking the condition of invariance of the set $K$ (which we have not done here), two different trajectories of the system \eqref{eq:lotka3d} have been simulated. The trajectories are illustrated in Fig. \ref{fig:gieslplot}.
\begin{figure}
\centering
\includegraphics[width=\columnwidth]{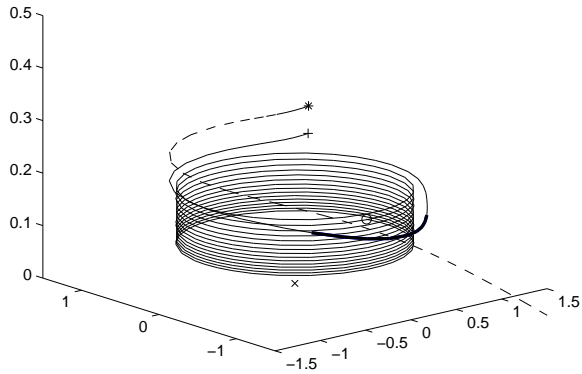}
\caption{Plot of two different trajectories of the system \eqref{eq:lotka3d}. The divergent trajectory marked with a dashed line has the initial condition $(x_0,y_0,z_0)=(1.03,1.03,0.25)$ which is marked with a $*$. The convergent trajectory marked with a solid line has the initial condition $(x_0,y_0,z_0)=(1,1,0.2)$ which is marked with a $+$. The portion of the latter trajectory which leaves the set $K$ is marked in fat and its final position is also marked with a $\circ$ in the plot. Finally, the equilibrium point at (0,0,0) is marked with a $\times$.}
\label{fig:gieslplot}
\end{figure}
As evident from Fig. \ref{fig:gieslplot} the two trajectories exhibit very different behaviour even though their initial conditions are very close (and contained in $K$). The solution marked with the dashed line escapes towards infinity while the trajectory marked with the solid line spirals towards the periodic orbit which is given by the unit circle in the $x,y$-plane. It is also worth noticing that the convergent trajectory leaves the set $K$ for a portion of time, which indicates that the basin of attraction of the periodic orbit is not fully contained in $K$.

To illustrate that the metric tensor $G$ has the desired properties, we have made two figures. In Fig. \ref{fig:eigG}, the minimum eigen value of $G$ along the convergent trajectory in Fig. \ref{fig:gieslplot} is given. As it can be seen from the figure, $G$ is positive definite along this particular trajectory. In Fig. \ref{fig:eig_contraction}, the maximal eigen value of $JG-\frac{1}{2}\dot{G}$ orthogonal to the vector field along the convergent trajectory in Fig. \ref{fig:gieslplot} is given. As it can be seen, the contraction condition is fulfilled along this particular trajectory. 
\begin{figure}
\centering
\includegraphics[width=0.8\columnwidth]{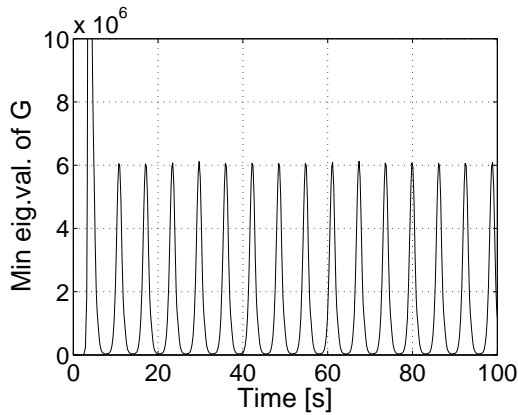}
\caption{Evaluation of the minimum eigen value of $G$ along the convergent trajectory in Fig. \ref{fig:gieslplot}.}
\label{fig:eigG}
\end{figure}
\begin{figure}
\centering
\includegraphics[width=0.8\columnwidth]{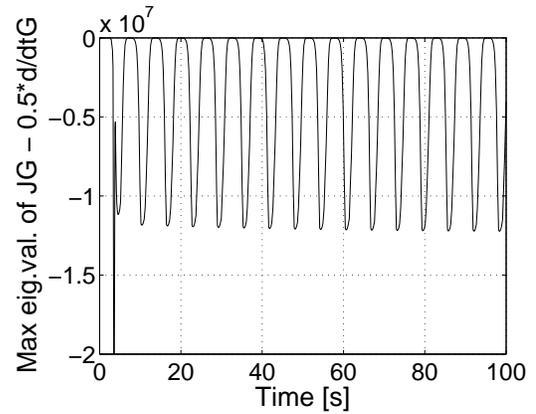}
\caption{Evaluation of the maximal eigen value of $JG-\frac{1}{2}\dot{G}$ orthogonal to the vector field along the convergent trajectory in Fig. \ref{fig:gieslplot}.}
\label{fig:eig_contraction}
\end{figure}

\section{Conclusion}
The purpose of this exposition was to find sufficient conditions for the existence and asymptotic stability of periodic orbits of nonlinear and polynomial dynamical systems. The conditions we were interested in should be numerically tractable, such that they could be checked algorithmically. We saw that in general one needs to verify two conditions: (i) invariance of some set $K\subset\mathbb{R}^n$ and (ii) the existence of some metric tensor $G$ such that the dynamical system fulfills a certain contraction property on the set $K$. Here, we focused on numerical methods for the construction of the metric tensor $G$, leaving the check of the condition (i) to existing numerical methods. We showed how the problem of constructing the metric tensor could be rewritten to a Sum-of-Squares (SOS) problem and used existing SOS software to solve the problem. The developed software tool was successfully tested on a numerical example of a system exhibiting an asymptotically stable periodic orbit. We also used the numerical example to illustrate the importance of checking both conditions (i) and (ii).

\section*{Acknowledgment}
This work was partially supported by the Danish Council for Strategic Research, under the Efficient Distribution of Green Energy (EDGE) research project.

\addtolength{\textheight}{-19.5cm}
\bibliographystyle{./IEEEtran}
\bibliography{./Contraction}

\end{document}